\documentclass[12pt]{article}

\usepackage[english]{babel}
\usepackage{graphicx}
\usepackage{latexsym}
\usepackage{amsfonts}
\usepackage{dsfont}
\usepackage{amscd}
\usepackage{amssymb}
\usepackage{amsmath}
\usepackage{hyperref}
\usepackage{mathrsfs}
\usepackage{amsthm}
\usepackage{tikz}

\input xy
\xyoption{all}

\newcommand{\NN}{\mathbb{N}}
\newcommand{\CC}{\mathbb{C}}
\newcommand{\ZZ}{\mathbb{Z}}

\newcommand{\I}{\mathcal{I}}
\newcommand{\p}{\mathcal{P}}
\newcommand{\DX}{\widetilde{\mathcal{D}}_\X}
\newcommand{\tX}{\widetilde{\X}}
\newcommand{\A}{\mathfrak{a}}
\newcommand{\X}{\mathsf{X}}
\newcommand{\U}{\mathcal{U}}
\newcommand{\s}{\mathcal{S}}
\newcommand{\OX}{\mathcal{O}_\X}

\newcommand{\Ef}{\widehat E_0}
\newcommand{\Eu}{\widehat E_\infty}
\newcommand{\Et}{\widehat E_{\text{tight}}}
\newcommand{\Ct}{C^*_{\text{tight}}}

\newcommand{\fs}{\subset_{\text{fin}}}
\newcommand{\kl}[2]{_{#1}{\sim}_{#2}}
\newcommand{\klclass}[3]{_{#1}\hspace{-0.03cm}#2_{#3}}

\theoremstyle{definition}
\newtheorem{theo}{Theorem}[section]
\newtheorem{rmk}[theo]{Remark}

\newtheorem{lem}[theo]{Lemma}
\newtheorem{defn}[theo]{Definition}
\newtheorem{cor}[theo]{Corollary}

\newtheorem{prop}[theo]{Proposition}
\date{}
\begin{document}
\title{Inverse semigroups associated to subshifts}
\author{Charles Starling\thanks{Supported by the NSERC grants of Beno\^it Collins, Thierry Giordano, and Vladimir Pestov. \texttt{cstar050@uottawa.ca}.}}
\maketitle

\begin{abstract}
The dynamics of a one-sided subshift $\mathsf{X}$ can be modeled by a set of partially defined bijections. From this data we define an inverse semigroup $\mathcal{S}_{\mathsf{X}}$ and show that it has many interesting properties. We prove that the Carlsen-Matsumoto C*-algebra $\mathcal{O}_\mathsf{X}$ associated to $\mathsf{X}$ is canonically isomorphic to Exel's tight C*-algebra of $\mathcal{S}_{\mathsf{X}}$. As one consequence, we obtain that $\mathcal{O}_\mathsf{X}$ can be written as a partial crossed product of a commutative C*-algebra by a countable group.
\end{abstract}

{\bf Author's note:} The results in this paper are not correct as stated. They become correct if one assumes that $\X$ satisfies Matsumoto's condition (I). A corrigendum has been submitted to J. Algebra, and has been appended to the end of this arXiv entry, before the references. 

\section{Introduction}
An {\em inverse semigroup} is a semigroup $S$ together with an involution $*: S\to S$ such that for all $s\in S$ we have 
\[
ss^*s = s.
\]
and such that $s^*$ is the only element for which this equation holds. On the other hand, a {\em partial isometry} in a C*-algebra $A$ is an element $v$ such that
\[
vv^*v = v.
\]
The link between inverse semigroups and partial isometries in C*-algebras implied by the above is hard to ignore, especially considering the mass of important examples of C*-algebras which are generated by partial isometries. In fact, for many C*-algebras of interest one can choose a countable generating set consisting of partial isometries which is also closed under product and adjoint -- such a set is necessarily an inverse semigroup.

On the one hand, C*-algebras have provided interesting examples of inverse semigroups. For examples of this, we can look to the graph inverse semigroups of \cite{Pa02} and \cite{La14}, the tiling inverse semigroup of \cite{KelCoin} and the AF inverse semigroups of \cite{LS14}. As one can imagine, each of these appears as a generating set of partial isometries in its namesake C*-algebra. 

On the other hand, if one knows a certain C*-algebra $A$ is generated by an inverse semigroup $S$, then (semigroup-theoretical) properties of $S$ give rise to properties of $A$. The catch is that $S$ may not tell the whole story. For example, the Cuntz algebra $\mathcal{O}_2$ and its Toeplitz extension $\mathcal{T}_2$ are both generated by inverse semigroups of partial isometries, and both of these inverse semigroups are isomorphic (as semigroups) to the same inverse semigroup $P_2$ (called a {\em polycyclic monoid} in the literature). What is happening is that there is a {\em representation} of $P_2$ in both algebras, but the relation $s_1s_1^* + s_2s_2^* = 1$ which holds in the Cuntz algebra cannot be expressed using only the multiplication and involution inside the inverse semigroup. Therefore, if one hopes to phrase simplicity (for example) of a C*-algebra in terms of properties of a generating inverse semigroup, one has to care for how the inverse semigroup is represented in the C*-algebra.\footnote{It turns out that another way of approaching this problem is to specialize further to a class of inverse semigroups called {\em boolean inverse monoids}. We do not pursue this here, and the interested reader is directed to \cite{LL13} for more details.}

This led Exel in \cite{Ex08} to define the notion of a {\em tight} representation of an inverse semigroup. He showed that for an inverse semigroup $S$ there always exists a C*-algebra, called the {\em tight} C*-algebra of $S$ and denoted $\Ct(S)$, which is universal for tight representations of $S$. It turns out that $\mathcal{O}_2$ is universal for tight representations of $P_2$, and $\mathcal{T}_2$ is universal for {\em all} representations of $P_2$ (a concept introduced by Paterson in \cite{Pa99}). Many C*-algebras of interest are isomorphic to the tight C*-algebra of their generating sets -- for example see \cite{EGS12} for tiling C*-algebras, \cite{Ex08} for graph and higher rank graph C*-algebras, \cite{StLCM} for boundary quotients of certain Cuntz-Li algebras, and \cite{EP17} for Katsura algebras and self-similar group algebras. 

Hence, there has been interest in relating the properties of an inverse semigroup to properties of its tight C*-algebra. The paper \cite{EP16} of Exel and Pardo provides conditions on $S$ which guarantee that $\Ct(S)$ is simple, and in the case that $\Ct(S)$ is nuclear these conditions become necessary and sufficient (invoking the results of \cite{BCFS14} regarding an underlying groupoid). They also give a condition which further guarantees that $\Ct(S)$ is purely infinite. Similar results are obtained in \cite{Ste16}.  Work of Milan and Steinberg \cite{MS14} gives conditions on $S$ which imply that $\Ct(S)$ is isomorphic to the partial crossed product of a commutative C*-algebra by a group, and further conditions which imply it is Morita equivalent to a usual crossed product.

With that setup, we turn to the present paper. We are concerned with one-sided subshifts over a finite alphabet $\A$, ie closed subspaces of $\A^\NN$ which are invariant under the left shift map. In \cite{Ma97}, Matsumoto associated a C*-algebra to such a space $\X$ which generalized the construction of Cuntz-Krieger C*-algebras \cite{CK80} (which can be naturally viewed as C*-algebras associated to shifts of finite type). In a subsequent paper with Carlsen \cite{CM04} a slightly different construction was put forward, and again in \cite{Ca08}. We will deal with the C*-algebra $\OX$ defined in \cite{Ca08}, and call this a {\em Carlsen-Matsumoto algebra}. This C*-algebra has been viewed under the lens of many different constructions: in \cite{Ca08} it is presented as a Cuntz-Pimsner algebra, in \cite{CS07} it is obtained as an Exel crossed product by an endomorphism, in \cite{CarlsenThesis} it is obtained from a Renault-Deaconu groupoid, and in \cite{Th10a} Thomsen constructs it from a semi-\'etale groupoid. 

Here, we add another construction to the list, by constructing an inverse semigroup $\s_\X$ from $\X$ and showing that $\OX$ is isomorphic to $\Ct(\s_\X)$. Our motivations for doing so are to add another example of a C*-algebra which can be seen as the tight C*-algebra of an inverse semigroup and also to provide another interesting example of an inverse semigroup arising from a C*-algebra. The other reason mentioned above that one might want to embark on such an investigation -- that properties of the algebra can be gleaned from that of the inverse semigroup -- is less pressing in this case, as Carlsen-Matsumoto algebras are already quite well-studied. For instance, the papers \cite{Th10a} and  \cite{CT12} combine to provide sharp conditions under which $\OX$ is simple and purely infinite. We do however use the results of \cite{MS14} to show that $\OX$ can be seen as a partial crossed product of a commutative C*-algebra by the free group over $\A$.

This paper is organized in the following manner. After providing some background, in Section 3 we define our inverse semigroup $\s_\X$ from $\X$, and show that it satisfies some nice properties. Section 4 is first devoted to establishing the isomorphism between $\OX$ and $\Ct(\s_\X)$, and then finishes by mentioning the partial crossed product result mentioned above.

\section{Preliminaries and notation}
We will use the following general notation. If $X$ is a set and $U\subset X$, let Id$_U$ denote the map from $U$ to $U$ which fixes every point, and let $1_{U}$ denote the characteristic function on $U$, ie $1_U: X\to \CC$ defined by $1_U(x) = 1$ if $x\in U$ and $1_U(x) = 0$ if $x\notin U$. If $F$ is a finite subset of $X$, we write $F\fs X$. We let $\NN$ denote the set of natural numbers (starting at 1).
\subsection{Inverse semigroups}\label{ISGsubsection}

An {\em inverse semigroup} is a semigroup $S$ such that for every $s\in S$, there exists a unique element $s^*\in S$, with the property that 
\[
ss^*s = s, \hspace{1cm}s^*ss^* = s^*.
\]
The element $s^*$ is called the {\em inverse} of $S$. For $s,t\in S$, we have $(s^*)^* = s$ and $(st)^* = t^*s^*$. We typically assume that $S$ has a neutral element $1$ and a zero element 0 such that
\[
1s = s1 = s\text{ for all }s\in S
\]
\[
0s = s0 = 0\text{ for all }s\in S.
\]
Even though we call $s^*$ the inverse of $s$, we need not have $ss^* = 1$, although we always have that $(ss^*)^2 = ss^*ss^* = ss^*$, which is to say that $ss^*$ (and indeed $s^*s$) is an {\em idempotent}. The set of all idempotents in $S$ is denoted
\[
E(S) = \{e\in S\mid e^2 = e\}.
\]
It is a nontrivial fact that if $S$ is an inverse semigroup, then $E(S)$ is closed under multiplication and commutative. It is also clear that if $e\in E(S)$, then $e^* = e$.

Let $X$ be a set, and let
\[
\I(X) = \{f:U\to V\mid U, V\subset X, f\text{ bijective}\}.
\]
Then $\I(X)$ is an inverse semigroup when given the operation of composition on the largest possible domain and inverse given by function inverse; it is called the {\em symmetric inverse monoid on $X$}.  If $e$ is an idempotent in $\I(X)$, then $e =$ Id$_U$ for some $U\subset X$. The function Id$_X$ is the neutral element for $\I(X)$, and the empty function is the 0 element for $\I(X)$. It is an important fact (akin to the Cayley theorem for groups) that every inverse semigroup is embeddable in $\I(X)$ for some set $X$ -- this is the Wagner-Preston theorem.

Every inverse semigroup possesses a natural order structure. For an inverse semigroup $s,t\in S$ we say $s\leqslant t$ if and only if $ts^*s = s$. On idempotents, this order has a nicer form -- if $e, f\in E(S)$ then $e\leqslant f$ if and only if $ef = e$. This partial order is perhaps best understood for elements of $\I(X)$, because if $g,h\in \I(X)$, then $g\leqslant h$ if and only if $h$ extends $g$ as a function.
\subsection{Subshifts}
As much as possible we use notation set in \cite{Ca08, CS07}. Let $\A$ be a finite set, called the {\em alphabet}, and endow it with the discrete topology. The product space
\[
\A^\NN = \prod_{n\in \NN}\A
\]
is called the {\em one-sided full shift over $\A$}. If $x=(x_n)_{n\in\NN}$, we will write $x$ in the shorter form
\[
x = x_1x_2x_3\cdots.
\]
The map $\sigma:\A^\NN\to \A^\NN$ given by $\sigma(x_1x_2x_3\cdots) = x_2x_3\cdots$ is called the {\em shift}, and is a continuous surjection. A subspace $\X\subset \A^\NN$ is called a {\em subshift} if it is closed and $\sigma(\X) \subset \X$. If this is the case, we will also sometimes say that $\X$ is a {\em one-sided subshift over $\A$}. Since $\A^\NN$ is compact and metrizable, then so is any subshift over $\A$.

For an integer $k\geq 1$, we let $\A^k$ denote the set of words of length $k$ in elements of $\A$ -- we again write an element $w\in \A^k$ as $w_1w_2\cdots w_k$. We also let $\A^0 = \{\epsilon\}$, and call $\epsilon$ the {\em empty word}. For $w\in\A^k$ we write $|w| =k$ and say that the {\em length} of $w$ is $k$.  We set $\A^* = \cup_{k\geq 0}\A^k$ and say this is the set of words in $\A$. Given $v,w\in\A^*$, we may form their concatenation
\[
vw = v_1v_2\cdots v_{|v|}w_1w_2\cdots w_{|w|}\in\A^*.
\]
In addition, for all $v\in \A^*$, we take $v\epsilon = \epsilon v = v$. Given $v\in \A^*$ and $x\in \A^\NN$, we may also concatenate $v$ and $x$:
\[
vx = v_1v_2\cdots v_{|v|}x_1x_2\cdots\in \A^\NN.
\]
Again, for all $x\in \A^\NN$ we let $\epsilon x = x$. If $v\in \A^*$, $x\in \A^*\cup \A^\NN$ and $y = vx$, then we say that $v$ is a {\em prefix} of $y$. For $x\in \X$ and $k\in\NN$ we let
\[
x_{[1,k]} = x_1x_2\cdots x_k,\hspace{1cm} x_{(k,\infty)} = x_{k+1}x_{k+2}\cdots.
\]
In addition, if $F\subset \A^*$ and $w\in \A^*$ we let $Fw = \{fw\mid f\in F\}$ and $wF = \{wf\mid f\in F\}$.

For $v\in\A^*$, we let $C(v) = \{vx\in\A^\NN\mid x\in \A^\NN\}$ and call sets of this form {\em cylinder sets}. These sets are closed and open in $\A^\NN$, and generate the topology on $\A^\NN$.  If $\X$ is a one-sided subshift over $\A$ and $v\in \A^*$, then we set $C_\X(v) = C(v)\cap \X$, although the subscript will frequently be dropped.

\section{Inverse semigroups associated to subshifts}
Given a one-sided subshift $\X$, the shift map $\sigma:\X\to \X$ is continuous, but in general it is not a local homeomorphism. Still, it is locally a {\em bijection}, in that $\left.\sigma\right|_{C(a)}$ is a bijection for all $a\in \A$. As mentioned in Section \ref{ISGsubsection}, inverse semigroups are a natural object with which to study partially defined bijections, and so we use the partial bijections above to associate an inverse semigroup $\s_\X$ to $\X$.
\subsection{Construction of $\s_\X$}
Following \cite{Ca08}, for $\mu, \nu\in \A^*$, we let
\[
C(\mu, \nu) = \{ \nu x \in \X\mid \mu x\in \X\}
\]
and notice that $C(\mu, \mu)= C(\mu)$. We note that
\[
C(\mu, \nu) = C(\nu)\cap\sigma^{-|\nu|}(\sigma^{|\mu|}(C(\mu))).
\]
Since the shift map is a closed map, $C(\mu, \nu)$ is closed for every $\mu, \nu\in\A^*$.

For each $a\in \A$, let $s_a\in \I(\X)$ be defined by 
\[
s_a: C(a, \epsilon) \to C(a, a)
\]
\[
s_a(x) = ax.
\]
For $\mu\in \A^*\setminus\{\epsilon\}$, we define $s_\mu = s_{\mu_1}s_{\mu_2}\cdots s_{\mu_{|\mu|}}$ so that
\[
s_\mu: C(\mu, \epsilon) \to C(\mu)
\]
\[
s_\mu(x) = \mu x.
\]
For the empty word $\epsilon$, we take $s_\epsilon =$ Id$_\X$. It is clear that for each $\mu\in\A^*$, the map $s_\mu$ is a bijection between subsets of $\X$.
\begin{defn}
Let $\X$ be a one-sided subshift over $\A$. Then we let $\s_\X$ be the inverse semigroup generated by $\{s_\epsilon, s_a\mid a\in \A\}$ inside $\I(\X)$.
\end{defn}

We would like to find a convenient closed form for elements of $\s_\X$. To this end, for each $F\fs \A^*$ and $\nu\in \A^*$, let\footnote{Late in preparation for this work, we discovered that sets of this form were already considered in \cite{Th10a}, and were written $C'(\nu;F)$. We use our notation in solidarity with \cite{Ca08} and keep the ``prefix'' data in the second entry and the ``possible replacement prefixes'' data in the first entry.}
\begin{eqnarray*}
C(F; \nu) &=& \{\nu x\in \X\mid fx\in \X\text{ for all }f\in F\}\\
          &=& \bigcap_{f\in F}C(f, \nu),
\end{eqnarray*}
and let $E(F;\nu) =$ Id$_{C(F;\nu)}$. We will also let $E(\mu,\nu) =$ Id$_{C(\mu, \nu)}$. A short calculation shows that
\[
E(\mu, \nu) = s_\nu s^*_\mu s_\mu s^*_\nu
\]
\[
E(F;\nu) = s_\nu \left(\prod_{f\in F}s^*_f s_f \right) s^*_\nu.
\]
The collection of all such elements will be important in the sequel -- we use the notation
\begin{equation}\label{EXdef}
E_\X = \{E(F;v)\in \I(\X)\mid v\in \A^*,  F\subset_{\text{fin}}\A^*\}\cup\{\emptyset\}.
\end{equation}
We note that the identity function on $\X$ is an element of $E(\X)$, taking $E(F;v)$ with $v = \epsilon$ and $F = \{\epsilon\}$. We also note that $E(F;v)E(G;w)\neq 0$ if and only if $C(F;v)\cap C(G;w)\neq \emptyset$.
\begin{lem}\label{idempotentlemma}
If $\X$ is a one-sided subshift over $\A$, then the set $E_\X$ is closed under multiplication. Furthermore, if $w\in \A^*$ and $e\in E_{X}$, then $s_w^*es_w\in E_\X$.
\end{lem}
\begin{proof}
Suppose that $F, G\fs\A^*$ and that $v, w\in \A^*$. Then
\begin{equation}\label{idempotentproduct}
E(F;v)E(G;w) = s_v \left(\prod_{f\in F}s^*_f s_f \right) s^*_vs_w \left(\prod_{g\in G}s^*_g s_g \right) s^*_w
\end{equation}
This product will be 0 unless $v$ is a prefix of $w$ or vice-versa. If $w = vz$, then
\begin{eqnarray*}
E(F;v)E(G;w) & = & s_v \left(\prod_{f\in F}s^*_f s_f \right) s^*_vs_vs_z \left(\prod_{g\in G}s^*_g s_g \right) s^*_{vz}\\
             & = & s_vs_z s^*_z\left(\prod_{f\in F}s^*_f s_f \right) s_zs^*_zs^*_vs_vs_z \left(\prod_{g\in G}s^*_g s_g \right) s^*_{vz}\\
             & = & s_w\left(\prod_{f\in F}s^*_zs^*_f s_fs_z \right)s^*_ws_w \left(\prod_{g\in G}s^*_g s_g \right)s^*_w\\
             & = & s_w \left(\prod_{f\in F}s^*_{fz} s_{fz} \right)\left(\prod_{g\in G}s^*_g s_g \right)s_w^*\\
             & = & E(G\cup Fz; w).
\end{eqnarray*}
A similar calculation shows that if $v = wz$, then
\[
E(F;v)E(G;w) = E(F\cup Gz; v).
\]

Furthermore,
\[
s_w^*E(F;v)s_w = s_w^*s_v\left(\prod_{f\in F}s^*_f s_f \right) s^*_v s_w.
\]
This is zero unless $v$ is a prefix of $w$ or vice-versa. If $w = vz$, then 
\begin{eqnarray*}
s_w^*E(F;v)s_w & = & s_w^*s_v\left(\prod_{f\in F}s^*_f s_f \right) s^*_v s_w\\
               & = & s_z^*s_v^*s_v\left(\prod_{f\in F}s^*_f s_f \right)s^*_vs_vs_z\\
               & = & s_z^*\left(\prod_{f\in F}s^*_f s_f \right)s_v^*s_vs_z\\
               & = & \left(\prod_{f\in F}s_z^*s^*_f s_fs_z \right)s_z^*s_v^*s_vs_z\\
               & = & E(\{w\}\cup Fz; \epsilon)
\end{eqnarray*}
If $v = wz$, then
\begin{eqnarray*}
s_w^*E(F;v)s_w & = & s_w^*s_ws_z\left(\prod_{f\in F}s^*_f s_f \right) s^*_zs^*_w s_w\\
               & = & s_z\left(\prod_{f\in F}s^*_f s_f \right)s_z^* s^*_ws_w\\
               & = & E(F;z)E(w, \epsilon)\\
               & = & E(F\cup \{wz\}; z)
\end{eqnarray*}
where the last line is by our previous calculation.
\end{proof}

\begin{prop}\label{SXprop}
Let $\X$ be a one-sided subshift over $\A$. Then
\begin{equation}\label{SX}
\s_{\X} = \{ s_\alpha E(F; v)s^*_\beta\in \I(\X)\mid \alpha, \beta, v\in \A^*, F\subset_{\text{fin}} \A^*\}\cup\{0\}.
\end{equation}
\end{prop}

\begin{proof}
We note that the containment ``$\supseteq$'' is trivial, because each element of the right hand side is a finite product of elements from $\{s_\epsilon, s_a\mid a\in \A\}$. Hence, we will be done if we can show that the set on the right hand side of \eqref{SX} is itself an inverse semigroup, because the right hand side contains $\{s_\epsilon, s_a\mid a\in \A\}$, and $\s_\X$ is the smallest inverse semigroup containing these elements.

It is clear that \eqref{SX} is closed under inverses, so we only need to show that it is closed under multiplication. To this end, take $\alpha, \beta, \delta, \eta, v, w\in\A^*$ and $F, G\subset_{\text{fin}}\A^*$. The product $\left(s_\alpha E(F; v)s^*_\beta\right)(s_\delta E(G; w)s_\eta^*)$ will again only be nonzero if $\beta$ is a prefix of $\delta$ or vice-versa. If $\delta = \beta\gamma$, then 
\begin{eqnarray*}
(s_\alpha E(F; v)s^*_\beta)(s_\delta E(G, w)s_\eta^*) 
&=&s_\alpha E(F; v)s^*_\beta s_\beta s_\gamma E(G, w)s_\eta^*\\
&=&s_\alpha s_\gamma s_\gamma^* E(F\cup\{\beta v\}; v) s_\gamma E(G, w)s_\eta^*
\end{eqnarray*}
and so by Lemma \ref{idempotentlemma}, this product is in $\s_\X$. A similar argument applies to the case that $\beta = \delta\gamma$. Hence $\s_\X$ is an inverse semigroup and we are done.
\end{proof}

\begin{rmk}\label{rmk:incorrectrmk}
For $s_\alpha E(F; v)s^*_\beta\in \s_\X$, if it happens that $\alpha_{|\alpha|}= \beta_{|\beta|} = a\in\A$, then 
\[
s_\alpha E(F; v)s^*_\beta = s_{\alpha_1\cdots\alpha_{|\alpha|-1}} E(F; av)s^*_{\beta_1\cdots\beta_{|\alpha|-1}}.
\]
For this reason, when working with elements of $\s_\X$ we will usually assume they are in ``lowest terms''. To be more precise, we take the following form for $\s_\X$:
\begin{equation}\label{SXdef}
\s_\X = \{s_\alpha E(F; v)s^*_\beta \mid \alpha, \beta, v\in \A^*, F\subset_{\text{fin}} \A^*, \alpha_{|\alpha|}\neq \beta_{|\beta|}\}\cup\{0\}.
\end{equation}
If we take two such elements $s_\alpha E(F; v)s^*_\beta$ and $s_\delta E(G; w)s_\eta^*$ with $\alpha_{|\alpha|}\neq \beta_{|\beta|}$ and $\delta_{|\delta|}\neq \eta_{|\eta|}$, then 
\[
s_\alpha E(F; v)s^*_\beta = s_\delta E(G; w)s_\eta^* \Rightarrow \alpha = \delta, \beta = \eta.
\]
We caution that the above equality does not imply that $E(F; v) = E(G; w)$, Indeed, a short calculation shows that 
\[
s_\alpha E(F; v)s^*_\beta = s_\alpha E(F\cup\{\alpha v, \beta v\}; v)s^*_\beta.
\]
We could put a condition on $\s_\X$ similar to \eqref{SXdef} stating that we assume $F$ contains $\alpha v$ and $\beta v$ when writing $s_\alpha E(F; v)s^*_\beta$, but this will usually not be necessary.
\end{rmk}
In the proof of Proposition \ref{SXprop} we started computation of the product of two elements of $\s_\X$, but stopped when it became clear that the product was again back in $\s_\X$. In the following lemma, we record the details of the exact form of this product.
\begin{lem}\label{productcomputation}
Let $\X$ be a one-sided subshift over $\A$, and take $\alpha, \beta, \delta, \eta, v, w\in\A^*$ and $F, G\subset_{\text{fin}}\A^*$.
\begin{enumerate}
\item If $\delta = \beta\gamma$ and $\gamma = vz$ for some $\gamma, z\in\A^*$, then
\begin{equation}
(s_\alpha E(F; v)s^*_\beta)(s_\delta E(G; w)s_\eta^*) = s_{\alpha\gamma} E(Fzw\cup G \cup\{\gamma w \}\cup \{\delta w\} ; w)s^*_\eta
\end{equation}
\item If $\delta = \beta\gamma$, $v = \gamma z$, and $z = wr$ for some $\gamma, z, r\in\A^*$, then
\begin{equation}
(s_\alpha E(F; v)s^*_\beta)(s_\delta E(G; w)s_\eta^*) = s_{\alpha\gamma} E(F\cup Gr \cup\{\beta v\}; z) s^*_\eta
\end{equation}
\item If $\delta = \beta\gamma$, $v = \gamma z$, and $w = zr$ for some $\gamma, z, r\in\A^*$, then
\begin{equation}
(s_\alpha E(F; v)s^*_\beta)(s_\delta E(G; w)s_\eta^*) = s_{\alpha\gamma} E(Fr\cup G \cup\{\beta vr\}; w) s^*_\eta
\end{equation}
\item If $\beta = \delta\gamma$ and $\gamma = wz$ for some $\gamma, z\in \A^*$ then
\begin{equation}
(s_\alpha E(F; v)s^*_\beta)(s_\delta E(G; w)s_\eta^*) = s_{\alpha} E(F\cup Gzv \cup\{\gamma v\}\cup\{\beta v\}; v) s^*_{\eta\gamma}
\end{equation}
\item If $\beta = \delta\gamma$, $w = \gamma z$, and $z = vr$ for some $\gamma, z, r\in \A^*$, then
\begin{equation}
(s_\alpha E(F; v)s^*_\beta)(s_\delta E(G; w)s_\eta^*) = s_{\alpha} E(Fr\cup G \cup\{\delta w\}; z) s^*_{\eta\gamma}
\end{equation}
\item If $\beta = \delta\gamma$, $w = \gamma z$, and $v = zr$ for some $\gamma, z, r\in \A^*$, then
\begin{equation}
(s_\alpha E(F; v)s^*_\beta)(s_\delta E(G; w)s_\eta^*) = s_{\alpha} E(F\cup Gr \cup\{\delta wr\}; v) s^*_{\eta\gamma}
\end{equation}
\item If none of the conditions in 1--6 above hold, then $(s_\alpha E(F; v)s^*_\beta)(s_\delta E(G; w)s_\eta^*) = 0$.
\end{enumerate}
\end{lem}
\begin{proof}
This follows from Lemma \ref{idempotentlemma}, and is left to the enthusiastic reader.
\end{proof}
\begin{lem}\label{rangeandsource}
Let $\X$ be a one-sided subshift over $\A$, let $\alpha, \beta, v\in\A^*$ and $F\fs \A^*$. Then
\[
(s_\alpha E(F; v)s^*_\beta)(s_\alpha E(F; v)s^*_\beta)^* = E(F\cup\{\beta v\};\alpha v)
\]
\[
(s_\alpha E(F; v)s^*_\beta)^*(s_\alpha E(F; v)s^*_\beta) = E(F\cup\{\alpha v\};\beta v).
\]
\end{lem}
\begin{proof}
This follows from Lemma \ref{productcomputation}.
\end{proof}
\begin{prop}
Let $\X$ be a one-sided subshift over $\A$, let $\s_\X$ be as in \eqref{SXdef}, and let $E_\X$ be as in \eqref{EXdef}. Then $E(\s_\X) = E_\X$.
\end{prop}
\begin{proof}
This follows from Lemma \ref{rangeandsource} together with the fact that the set of idempotents of an inverse semigroup $S$ coincides with the set of elements of the form $s^*s$ for $s\in S$.

\end{proof}

\begin{rmk}
A recent preprint of Boava, de Castro, and Mortari \cite{BdCM15} associates an inverse semigroup to every {\em labeled space}. In \cite{BCP12} Bates, Carlsen, and Pask associate a labeled space to any one-sided subshift $\X$, such that the C*-algebra associated to the constructed labeled space is isomorphic to $\OX$, see \cite[Example 4]{BCP12}. We caution that the inverse semigroup that one obtains by combining \cite{BdCM15} and \cite{BCP12} (say, $\tilde\s_\X$) will not be the same as our $\s_\X$ -- in fact $E(\tilde\s_\X)$ will be isomorphic to the Boolean algebra generated by the $C(v,w)$ as $v$ and $w$ range over $\A^*$. Hence, their set of idempotents will contain complements of the $C(v,w)$ while ours (in general) will not. For the specific situation of a subshift $\X$ our construction seems  natural, as the only idempotents which appear in our construction are those which arise directly from the partial bijections arising from the shift map on $\X$.
\end{rmk}
\subsection{Properties of $\s_\X$}
We now discuss some useful properties which our newly-defined inverse semigroup $\s_\X$ may possess.
\begin{defn}
Let $S$ be an inverse semigroup with identity and zero (in other words, an inverse monoid with zero). 
\begin{enumerate}\label{ISGproperties}
\item We say that $S$ is {\em $E^*$-unitary} if $0 \neq e \leqslant s$ with $e\in E(S)$ implies that $s\in E(S)$.
\item If $\Gamma$ is a group and $\phi: S\setminus\{0\}\to \Gamma$ such that $s, t\in S$ with $st\neq 0$ implies that $\phi(st) = \phi(s)\phi(t)$, then we say that $\phi$ is a {\em partial homomorphism} from $S$ to $\Gamma$. If, in addition, $\phi^{-1}(1_\Gamma) = E(S)$, then $\phi$ is called an {\em idempotent pure partial homomorphism}.
\item We say that $S$ is {\em strongly $E^*$-unitary} if there exists a group $\Gamma$ and an idempotent-pure partial homomorphism from $S$ to $\Gamma$.
\item We say that $S$ is {\em F*-inverse} if for each $s\in S$ there exists a unique maximal element above $s$.
\item We say that $S$ is {\em strongly $F^*$-inverse} if there exists a group $\Gamma$ and an idempotent-pure partial homomorphism $\phi$ from $S$ to $\Gamma$ such that for all $g\in \Gamma$, $\phi^{-1}(g)$ has a maximal element whenever it is nonempty.
\end{enumerate}
\end{defn}

\begin{lem}
Let $\X$ be a one-sided subshift over $\A$, and let $\s_\X$ be as in \eqref{SXdef}. Then $\s_\X$ is $E^*$-unitary.
\end{lem}
\begin{proof}
We note that for an idempotent $e$, $e \leqslant s$ if and only if $se = e$. Suppose that $\alpha, \beta, v, w\in \A^*$, that $F, G \fs \A^*$, that $\alpha_{|\alpha|}\neq \beta_{|\beta|}$, and that $E(G;w), s_\alpha E(F; v) s_\beta^* \neq 0$. We have
\begin{eqnarray*}
(s_\alpha E(F; v) s_\beta^*)E(G;w) & = & s_\alpha E(F; v)(s_\beta^*E(G;w)s_\beta) s_\beta^*\\
                                   & = & s_\alpha E(F'; v')s_\beta^*\hspace{0.5cm}\text{for some }F'\fs \A^*, v'\in\A^*
\end{eqnarray*}
If this is equal to $E(G;w)$, we must have that $\alpha = \beta$. Since the last letters of $\alpha, \beta$ were assumed to be unequal, this implies that $\alpha = \beta = \epsilon$, and hence $s_\alpha E(F; v) s_\beta^*$ is an idempotent as required.
\end{proof}
We now prove that $\s_\X$ is strongly $E^*$-unitary, which seems to make the above lemma a waste because evidently being strongly $E^*$-unitary implies being $E^*$-unitary. Still, we believe that the above lemma is instructive, so there it stays.
\begin{lem}\label{stronglyEunitarylemma}
Let $\X$ be a one-sided subshift over $\A$, and let $\s_\X$ be as in \eqref{SXdef}. Then $\s_\X$ is strongly $E^*$-unitary.
\end{lem}
\begin{proof}
Let $\mathbb{F}_\A$ denote the free group on the alphabet $\A$. For $\alpha, \beta, v\in \A^*$, $F\fs \A^*$ such that $\alpha_{|\alpha|}\neq \beta_{|\beta|}$, we define a map $\phi: \s_\X\setminus\{0\}\to \mathbb{F}_\A$ by 
\[
\phi(s_\alpha E(F; v) s_\beta^*) = \alpha\beta^{-1}.
\]
We claim that this map is a partial homomorphism. To prove this, we take  $\alpha, \beta, \delta, \eta, v, w\in\A^*$ and $F, G\subset_{\text{fin}}\A^*$ such that $\alpha_{|\alpha|}\neq \beta_{|\beta|}$, $\delta_{|\delta|}\neq \eta_{|\eta|}$ and suppose that $(s_\alpha E(F; v)s^*_\beta)(s_\delta E(G; w)s_\eta^*) \neq 0$. This implies that either $\delta = \beta\gamma$ or $\beta = \delta\gamma$ for some $\gamma \in \A^*$. 

If $\delta = \beta\gamma$ for some $\gamma \in \A^*$, then 
\[
\phi(s_\alpha E(F; v)s^*_\beta)\phi(s_\delta E(G; w)s_\eta^*) = \alpha\beta^{-1}\delta\eta^{-1} = \alpha\beta^{-1}\beta\gamma\eta^{-1} = \alpha\gamma\eta^{-1}.
\]
On the other hand, in each of the first three cases of Lemma \ref{productcomputation}, the product of these two elements is $s_{\alpha\gamma} A s_{\beta}^*$ for some $A\in E_\X$. Hence $\phi(s_\alpha E(F; v)s^*_\beta s_\delta E(G; w)s_\eta^*) = \alpha\gamma\eta^{-1}$. The case $\beta = \delta\gamma$ is similar. Hence $\phi$ is a partial homomorphism.

Furthermore, if $\phi(s_\alpha E(F; v) s_\beta^*) = \alpha\beta^{-1} = 1_{\mathbb{F}_\A}$, then $\alpha = \beta$, and as before this implies that $s_\alpha E(F; v) s_\beta^* = E(F;v)$, an idempotent. Thus, $\phi$ is idempotent pure.
\end{proof}
Finally, we consider the last two properties from Definition \ref{ISGproperties}.
\begin{lem}\label{maximalelementlemma}
Suppose that $\alpha, \beta, v\in \A^*$, that $F\fs \A^*$, and that $\alpha_{|\alpha|}\neq \beta_{|\beta|}$. Then $s_\alpha E(F; v) s_\beta^* \leqslant s_\alpha s_\beta^*$. Furthermore, if $s\in \s_\X$ and $s_\alpha E(F; v) s_\beta^* \leqslant s$, then $s\leqslant s_\alpha s_\beta^*$.
\end{lem}
\begin{proof}
Let $t = s_\alpha E(F; v) s_\beta^*$. We first must show that $s_\alpha s_\beta^* t^* t = t$. By Lemma \ref{rangeandsource}, $t^*t = E(F\cup\{\alpha v\}; \beta v)$. We calculate
\begin{eqnarray*}
s_\alpha s_\beta^* t^*t & = & s_\alpha s_\beta^* E(F\cup\{\alpha v\}; \beta v)\\
                        & = & s_\alpha s_\beta^* s_\beta E(F\cup\{\alpha v\}; v)s_\beta^*\\
                        & = & s_\alpha E(F\cup\{\alpha v\}; v) s_\beta^* s_\beta s_\beta^*\\
                        & = & s_\alpha E(F\cup\{\alpha v\}; v) s_\beta^*\\
                        & = & s_\alpha s_v s_v^* s_\alpha^* s_\alpha s_v \left(\prod_{f\in F}s_f^*s_f\right)s_v^* s_\beta^*\\
                        & = & s_\alpha s_v \left(\prod_{f\in F}s_f^*s_f\right)s_v^* s_\beta^*\\
                        & = & s_\alpha E(F; v)s_\beta^*
\end{eqnarray*}
Now, take $\delta, \eta, \in \A^*$ with $\delta_{|\delta|}\neq \eta_{|\eta|}$, $A\in E_\X$ and let $s = s_\delta A s_\eta^*$. Then
\[
s t^* t = s_\delta A s_\eta^* E(F\cup\{\alpha v\}; \beta v) = s_\delta A s_\eta^* E(F\cup\{\alpha v\}; \beta v)s_\eta s_\eta^* =  s_\delta B s_\eta^*
\]
for some $B\in E_\X$ by Lemma \ref{idempotentlemma}. If this is equal to $s$, then $\delta = \alpha$ and $\eta = \beta$. Thus by the above calculation, we must have $s\leqslant s_\alpha s_\beta^*$ as well.
\end{proof}

We can now prove the following.

\begin{prop}
Let $\X$ be a one-sided subshift over $\A$ and let $\s_\X$ be as in \ref{SXdef}. Then $\s_\X$ is strongly $F^*$-inverse.
\end{prop}
\begin{proof}
Let $\phi$ be as defined in the proof of Lemma \ref{stronglyEunitarylemma}. Then if $\phi^{-1}(g)$ is not empty, $g = \alpha\beta^{-1}$ for some $\alpha, \beta\in \A^*$, and as in the proof of Lemma \ref{maximalelementlemma}, $s_\alpha s_\beta^*$ is maximal in $\phi^{-1}(\alpha\beta^{-1})$.
\end{proof}
\section{C*-algebras}

We now turn our attention to the C*-algebras associated to the structures we have defined. The main result of this section, Theorem \ref{mainresult}, states that given a one-sided subshift $\X$, a certain C*-algebra $\OX$ associated to $\X$ is canonically isomorphic to a certain C*-algebra $\Ct(\s_\X)$ associated to the inverse semigroup $\s_\X$. We first recall the construction of $\OX$ due to Matsumoto and Carlsen, and then the construction of $\Ct(S)$ for a general inverse semigroup $S$. Knowledge of C*-algebras is assumed -- one can find undefined terms in the excellent reference \cite{DAV}.
\subsection{The Carlsen-Matsumoto algebra $\OX$}\label{CMalgebrasection}
Let $\X$ be a one-sided subshift over $\A$, and consider $\ell^\infty(\X)$, the C*-algebra of bounded functions on $\X$. Define $\DX$ to be the C*-subalgebra of $\ell^\infty(\X)$ generated by $\{1_{C(\mu, \nu)} \mid \mu, \nu\in\A^*\}$. We can now define the algebra $\OX$. 

\begin{defn}(See \cite[Theorem 10]{CS07}) \label{OXdef}
Let $\X$ be a one-sided subshift over $\A$. Then the {\em Carlsen-Matsumoto algebra} $\OX$ is the universal C*-algebra generated by partial isometries $\{S_\mu\}_{\mu\in \A^*}$ such that
\begin{enumerate}
\item $S_\mu S_\nu = S_{\mu\nu}$ for all $\mu, \nu\in \A^*$, and
\item The map $1_{C(\mu, \nu)}\mapsto S_\nu S_\mu^* S_\mu S_\nu^*$ extends to a $*$-homomorphism from $\DX$ to the C*-algebra generated by $\{S_\mu\mid \mu\in \A^*\}$.
\end{enumerate}
\end{defn}
So $\OX$ is generated by a set of partial isometries $\{S_\mu\}_{\mu\in\A^*}$, and we view $\DX$ as the subalgebra of $\OX$ generated by elements of the form $S_\nu S_\mu^* S_\mu S_\nu^*$. One can show that $\OX$ is unital, with unit $I_{\OX} = I_{\DX} = S_\epsilon$. Furthermore, one can show that the elements $\{S_\mu\}_{\mu\in\A^*}$ satisfy
\begin{equation}\label{cuntz}
\sum_{a\in\A}S_aS_a^* = I_{\OX},
\end{equation}
\begin{equation}\label{comm1}
S_\mu^* S_\mu S_\nu S_\nu^* = S_\nu S_\nu^* S_\mu^* S_\mu,
\end{equation}
\begin{equation}\label{comm2}
S_\mu^* S_\mu S_\nu^* S_\nu = S_\nu^* S_\nu S_\mu^* S_\mu.
\end{equation}
In addition, if $\mu, \nu\in \A^*$ with $|\mu|=|\nu|$, then
\begin{equation}\label{ortho}
S^*_\mu S_\nu \neq 0 \Rightarrow \mu = \nu.
\end{equation}

Since $\DX$ is a commutative C*-algebra, it is isomorphic to $C(\tX)$ for a certain compact Hausdorff space $\tX$. This space was presented as an inverse limit space in \cite[Chapter 2]{CarlsenThesis}, and we reproduce this presentation here because we will use it to establish an isomorphism between $\Ct(\s_\X)$ and $\OX$.

For $x\in \X$ and integer $k\geq 0$, let
\[
\p_k(x)= \{\mu\in \A^*\mid \mu x\in X, |\mu| = k\}
\]
For $l\in \NN$, we say that $x, y\in \X$ are {\em $l$-past equivalent} and write $x\sim_l y$ if $\p_k(x) = \p_k(y)$ for all $k\leq l$. The $l$-past equivalence class of $x\in\X$ will be written as $[x]_l$.

Let $\I = \{ (k, l)\in \NN^2 \mid k \leq l\}$. For every $(k, l)\in \I$ we define another equivalence relation $\kl{k}{l}$ on $\X$ by
\[
x\kl{k}{l}y \Leftrightarrow x_{[1,k]} = y_{[1,k]} \text{ and }\p_r(x_{(k,\infty)}) = \p_r(y_{(k,\infty)})\text{ for all }r\leq l.
\]
We note that there is a typo in \cite[Chapter 2]{CarlsenThesis} where the above is defined with $r=l$ rather than $r\leq l$.\footnote{This was confirmed in private communication with Carlsen.} The equivalence class of $x\in\X$ under $\kl{k}{l}$ will be written as $\klclass{k}{[x]}{l}$, and the set of all such equivalence classes will be written as $\klclass{k}{\X}{l}$. It is clear that for all $(k,l)\in \I$, $\klclass{k}{\X}{l}$ is finite; we endow it with the discrete topology.

There is a partial order on $\I$ which respects this equivalence relation. For $(k_1, l_1), (k_2, l_2)\in\I$ we say 
\[
(k_1, l_1)\leq (k_2, l_2) \Leftrightarrow k_1\leq k_2\text{ and }l_1 - k_1 \leq l_2-k_2.
\]
We note that if $(k,l), (r,s)\in \I$, then they have a common upper bound. Indeed, if $k =r$ then $(k, \max\{l,s\})$ is an upper bound for $(k,l)$ and $(r,s)$, and if $k<r$ then $(r, \max\{l+r-k, s\})$ is an upper bound for $(k,l)$ and $(r,s)$. If $(k_1, l_1)\leq (k_2, l_2)$ then it is straightforward that
\[
x \kl{k_2}{l_2}y \Rightarrow x\kl{k_1}{l_1}y.
\]
Thus, for $(k_1, l_1)\leq (k_2, l_2)$, there is a map $\klclass{(k_1, l_1)}{\pi}{(k_2, l_2)}:\ \klclass{k_2}{\X}{l_2}\to \ \klclass{k_1}{\X}{l_1}$ such that $$\klclass{(k_1, l_1)}{\pi}{(k_2, l_2)}(\klclass{k_2}{[x]}{l_2}) =\  \klclass{k_1}{[x]}{l_1}$$
One can then form the inverse limit 
\begin{eqnarray}
\tX &=& \lim_{(k, j)\in \I}(\ \klclass{k}{\X}{l}, \pi)\label{tXdef}\\
 &=& \left\{ (\klclass{k}{[\klclass{k}{x}{l}]}{l})_{(k,l)\in\I} \in \prod_{(k,l)\in\I}\ \klclass{k}{\X}{l}\mid (k_1, l_1)\leq (k_2, l_2) \Rightarrow\ \klclass{k_1}{[\klclass{k_2}{x}{l_2}]}{l_1} =\ \klclass{k_1}{[\klclass{k_1}{x}{l_1}]}{l_1}\right\}\nonumber
\end{eqnarray}
Which is a closed subspace of the space $\prod_{(k,l)\in\I}\ \klclass{k}{\X}{l}$ when given the product topology of the discrete topologies.

Let $x\in \X$ and take $(k,l)\in \I$. The set
\[
U(x, k, l) = \{(\klclass{r}{[\klclass{r}{x}{s}]}{s})_{(r,s)\in\I} \in \tX \mid\ \klclass{k}{[\klclass{k}{x}{l}]}{l} =\ \klclass{k}{[x]}{l}\}
\]
is open and closed. Sets of this form generate the topology on $\tX$.

We have the following lemma about the relation $\kl{k}{l}$.

\begin{lem}\label{eqclasscontain}
Let $\X$ be a subshift, let $v\in \A^*$, let $F\fs \A^*$, and let
\[
k = |v|, \hspace{1cm} l = \max\{|f|, |v| :  f\in F\}.
\]
Then for all $x\in \X$, and all $(r,s)\geq (k,l)$, either $\klclass{r}{[x]}{s} \subset C(F;v)$ or $\klclass{r}{[x]}{s} \cap C(F; v) = \emptyset$.
\end{lem}
\begin{proof}
Since $(r,s)\geq (k,l)$ implies that $\klclass{r}{[x]}{s} \subset$ $\klclass{k}{[x]}{l}$, we need only prove the statement for $r=k$ and $s = l$. Suppose that we have $\klclass{k}{[x]}{l} \not\subset C(F;v)$, and take $y\in\ \klclass{k}{[x]}{l} \setminus C(F;v)$. If $v$ is not a prefix of $y$, then this is true of all elements of $\klclass{k}{[x]}{l}$ and we are done. So, suppose that $y = vy'$. There must be an element $f\in F$ such that $fy'\notin \X$. If we have some other element $z = vz'\in\ \klclass{k}{[x]}{l}$, we must have that $\p_{|f|}(z') = \p_{|f|}(y')$, and so $fz'\notin \X$. This implies that $z\notin C(F;v)$ and we are done.
\end{proof}
\subsection{$\OX$ as the tight C*-algebra of $\s_\X$}
In this section we recall the definition of the tight C*-algebra of an inverse semigroup from \cite{Ex08}. We then show that the tight C*-algebra of $\s_\X$ is isomorphic to $\OX$.

Let $S$ be an inverse semigroup with 0, and let $A$ be a C*-algebra. A map $\pi: S\to A$ is called a {\em representation} of $S$ if $\pi(0) = 0$, $\pi(st) = \pi(s)\pi(t)$ and $\pi(s)^* = \pi(s^*)$ for all $s, t\in S$.

We are interested in a certain class of representations which we will now describe. For $F\subset Z\subset E(S)$, we say that $F$ {\em covers} $Z$ if for every $z\in Z$, there exists $f\in F$ such that $fz \neq 0$. If $F$ covers $\{y\in E(S)\mid y\leqslant x\}$, we say that $F$ covers $x$. 

Let $X,Y\fs E(S)$, and let
\[
E(S)^{X, Y} = \{e\in E(S)\mid e\leqslant x \text{ for all }x\in X, ey = 0 \text{ for all }y\in Y\}.
\]
A representation $\pi: S\to A$ with $A$ unital is said to be {\em tight} if whenever $X, Y, Z\fs E(S)$ such that $Z$ is a cover of $E(S)^{X,Y}$, then
\[
\bigvee_{z\in Z}\pi(z) = \prod_{x\in X}\pi(x)\prod_{y\in Y}(1-\pi(y)).
\]

The {\em tight C*-algebra of $S$}, denoted $\Ct(S)$, is the universal C*-algebra generated by one element for each element of $S$ subject to the relations which say that the standard map $\pi_t: S\to \Ct(S)$ is a tight representation.

At this point it is not clear that $\Ct(S)$ exists, but it was explicitly constructed in \cite{Ex08} as a groupoid C*-algebra associated to an action of $S$ on a certain space $\Et(S)$ associated to $S$. We do not go into specifics about inverse semigroup actions or groupoids here, though we will define $\Et(S)$ as it is essential for establishing isomorphism we desire. 

Recall that the natural partial order on $S$, when restricted to $E(S)$, takes on a simpler form: $e\leqslant f \Leftrightarrow ef = e$. A subset $\xi\subset E(S)$ is called a {\em filter} if it does not contain the zero element, is closed under products, and is ``upwards directed'', which is to say that if $e\in \xi$ and $e\leqslant f$ then $f\in \xi$. A filter is called an {\em ultrafilter} if it is not properly contained in any other filter. The set of filters is denoted $\Ef(S)$, and the set of ultrafilters is denoted $\Eu(S)$. 

The set $\Ef(S)$ may be viewed as a subset of the product space $\{0,1\}^{E(S)}$. We let $\Ef(S)$ inherit the subspace topology from the product topology (with $\{0,1\}$ given the discrete topology). For $e\in E(S)$, let
\[
D_e = \{\xi\in \Ef(S)\mid e\in \xi\}.
\]
Then sets of this form together with their complements form a subbasis for the topology on $\Ef(S)$. With this topology, $\Ef(S)$ is called the {\em spectrum} of $E(S)$. We also let $\Et(S) = \overline{\Eu(S)}$, and call this the {\em tight spectrum} $E(S)$. We will shorten $D_e\cap \Et(S)$ to $D_e^t$.

 It is a fact that $\Ct(S)$ exists and that the C*-subalgebra of $\Ct(S)$ generated by $\pi_t(E(S))$ is $*$-isomorphic to $C(\Et)$ via the identification $\pi_t(e) \mapsto 1_{D_e^t}$.

Now, we take $\X$ to be a one-sided subshift over $\A$, and describe $\Et(\s_\X)$.
\begin{lem}
If $\xi\subset E_\X$ is a filter, then there exists $x\in \X \cup \A^*$ such that if $E(F;v)\in \xi$, then $v$ is a prefix of $x$.
\end{lem}
\begin{proof}
If $\xi$ is a filter and $E(G;v), E(H, w)\in\xi$, then Lemma \ref{idempotentlemma} shows that their product is zero unless $w$ is a prefix of $v$ or vice-versa. The result follows.
\end{proof}
\begin{lem}
Let $\X$ be a one-sided subshift over $\A$, and let
\[
\eta_x = \{E(F; v)\in E_\X\mid x\in C(F; v)\}.
\]
Then $\Eu = \{\eta_x\mid x\in \X\}$.
\end{lem}
\begin{proof}
First, we show that $\eta_x$ is a filter. If $E(F; v), E(G, w)\in \eta_x$, then $C(F;v)\cap C(G;w) = C(H; z)\neq \emptyset$ for some $H\fs \A^*$ and $z\in \A^*$. Hence $E(F;v)E(G;w) = E(H;z)$, and so $\eta_x$ is closed under products. It is clear that $\eta_x$ does not contain the zero element and is upwards closed, so it is a filter.

Now, suppose that we have $E(F;v)$ such that $E(F;v)E(G;w)\neq 0$ for all $E(G;w)\in \eta_x$. Thus for each $n\geq 0$, we can find $y_n \in C(F;v)\cap C(x_1\cdots x_n)$, and it is clear that the $y_n$ converge to $x$ in $\X$. Since $C(F;v)$ is closed in $\X$, we must have that $x\in C(F;v)$, and so $E(F;v)\in \eta_x$. This shows that $\eta_x$ is an ultrafilter.

Now, suppose that $\xi\subset E_\X$ is an ultrafilter. Then $\{C(F;v)\mid E(F;v)\in \xi\}$ is a collection of closed subsets of the compact space $\X$ which has the finite intersection property, and so the intersection
\[
\bigcap_{E(F;v)\in \xi}C(F;v)
\]
is nonempty. Take $x$ in the above intersection. Then we must have that $\xi\subset \eta_x$, and since $\xi$ is assumed to be an ultrafilter, $\xi = \eta_x$.
\end{proof}

We now return to the space $\tX$ from \eqref{tXdef} which is the spectrum of the commutative C*-algebra $\DX$. Our next proposition will establish a natural homeomorphism between $\tX$ and $\Et(\s_\X)$.
\begin{prop}\label{xttight}
The map $\theta: \tX \to \Ef(\s_\X)$ defined by
\begin{equation}\label{thetadef}
\theta\left((\klclass{k}{[\klclass{k}{x}{l}]}{l})_{(k,l)\in\I}\right) = \{E(F;v)\in E_\X\mid\ \klclass{k}{[\klclass{k}{x}{l}]}{l}\subset C(F;v)\text{ for some }(k,l)\in\I\}
\end{equation}
is continuous, injective, and $\theta(\tX) = \Et(\s_\X)$. Hence, it is a homeomorphism from $\tX$ to $\Et(\s_\X)$.
\end{prop}
\begin{proof}
First we show that $\theta$ is well-defined. Take $(\klclass{k}{[\klclass{k}{x}{l}]}{l})_{(k,l)\in\I}\in \tX$ and consider its image under $\theta$ -- it is clearly upwards closed and does not contain the zero element. To prove closure under products, suppose we have $(r,s), (t, u)\in \I$ and $E(F; v), E(G; w)\in E_\X$ such that $\klclass{r}{[\klclass{r}{x}{s}]}{s} \subset E(F; v)$ and $\klclass{t}{[\klclass{t}{x}{u}]}{u} \subset E(G; w)$. Let $(k,l)$ be an upper bound for $(r, s), (t, u)$ in $\I$. Then $\klclass{k}{[\klclass{k}{x}{l}]}{l}$ is a subset of both $\klclass{r}{[\klclass{r}{x}{s}]}{s}$ and $\klclass{t}{[\klclass{t}{x}{u}]}{u}$, and so is contained in both $C(F;v)$ and $C(G;w)$. If $E(F;v)E(G;w) = E(H;z)$, then $\klclass{k}{[\klclass{k}{x}{l}]}{l}\subset C(H;z)$, and so $\theta((\klclass{k}{[\klclass{k}{x}{l}]}{l})_{(k,l)\in\I})$ is a filter.

We now show that $\theta$ is injective. Suppose we have $x, y\in \tX$ and that $x\neq y$. Then there must exist $(k,l)\in\I$ such that $\klclass{k}{[\klclass{k}{x}{l}]}{l} \neq \klclass{k}{[\klclass{k}{y}{l}]}{l}$. If $(\klclass{k}{x}{l})_{[1,k]} \neq (\klclass{k}{y}{l})_{[1,k]}$, then $$E\left(\bigcup_{r\leq l}\p_r((\klclass{k}{x}{l})_{(k,\infty)}); \ (\klclass{k}{x}{l})_{[1,k]}\right)\in \theta(x)$$ $$E\left(\bigcup_{r\leq l}\p_r((\klclass{k}{y}{l})_{(k,\infty)}); \ (\klclass{k}{y}{l})_{[1,k]}\right)\in \theta(y).$$
The product of these two elements is zero, so $\theta(x)\neq \theta(y)$.

So, we instead suppose that there exists $v\in \A^*$ with $|v| = k$ and $\klclass{k}{x}{l} = vx'$, $\klclass{k}{y}{l} = vy'$, so that
\[
\klclass{k}{[vx']}{l}\neq\ \klclass{k}{[vy']}{l}.
\]
Without loss of generality, there must exist $w\in \A^*$ with $|w|\leq l$ such that $wx'\in\X$ and $wy'\notin \X$. Hence, $vx'\in C(w,v)$, and $vy'\notin C(w,v)$. Thus by Lemma \ref{eqclasscontain}, we must have that $\klclass{k}{[vx']}{l} \subset C(w,v)$ and $\klclass{k}{[vy']}{l}\cap C(w,v) = \emptyset$. Similar to above, this implies that $E(w,v)\in \theta(x)$ and $E(w,v)\notin \theta(y)$. Hence $\theta(x)\neq\theta(y)$, and $\theta$ is injective.

Now, we prove that $\theta$ is continuous. Take $E(F;v)\in E_\X$, and as before take $D_{E(F;v)} = \{\xi\in \Ef(\s_\X)\mid E(F;v)\in \xi\}$. Then
\[
\theta^{-1}(D_{E(F;v)}) = \{(\klclass{k}{[\klclass{k}{x}{l}]}{l})_{(k,l)\in\I}\in \tX\mid\ \klclass{r}{[\klclass{r}{x}{s}]}{s} \subset C(F;v)\text{ for some }(r,s)\in\I\}.
\]
If $(\klclass{k}{[\klclass{k}{x}{l}]}{l})_{(k,l)\in\I}\in \theta^{-1}(D_{E(F;v)})$, find $(r,s)\in\I$ such that $\klclass{r}{[\klclass{r}{x}{s}]}{s} \subset C(F;v)$. Then if $(\klclass{k}{[\klclass{k}{y}{l}]}{l})_{(k,l)\in\I}\in U(\klclass{r}{x}{s},r,s)$, $\klclass{r}{[\klclass{r}{y}{s}]}{s} =\ \klclass{r}{[\klclass{r}{x}{s}]}{s} \subset C(F;v)$, and so $U(\klclass{r}{x}{s},r,s)\subset \theta^{-1}(D_{E(F;v)})$.

On the other hand,
\[
\theta^{-1}((D_{E(F;v)})^c) = \{(\klclass{k}{[\klclass{k}{x}{l}]}{l})_{(k,l)\in\I}\in \tX\mid\ \klclass{r}{[\klclass{r}{x}{s}]}{s} \not\subset C(F;v)\text{ for all }(r,s)\in\I\}.
\]
Take $(\klclass{k}{[\klclass{k}{x}{l}]}{l})_{(k,l)\in\I}\in \theta^{-1}((D_{E(F;v)})^c)$, take $k = |v|$, and take $l =  \max\{|f|, |v| :  f\in F\}$. Then $\klclass{k}{[\klclass{k}{x}{l}]}{l}\cap C(F;v) = \emptyset$, and so $U(\klclass{k}{x}{l}, k, l)\subset \theta^{-1}((D_{E(F;v)})^c)$. The collection of all sets of the form $D_{E(F;v)}$ together with those of the form $(D_{E(G;w)})^c$ form a subbasis for the topology on $\Ef$, so $\theta$ is continuous.

Finally, we must show that $\theta(\tX) = \Et(\s_\X)$. For $x\in \X$, let 
\[
\tilde x = (\klclass{k}{[x]}{l})_{(k,l)\in\I}.
\]
Because sets of the form $U(x, k,l)$ for $x\in \X$ and $(k,l)\in\I$ form a basis for the topology on $\tX$, the set $\{\tilde x\in \tX\mid x\in\X\}$ is dense in $\tX$.

We claim that $\theta(\tilde x) = \eta_x$. If $E(F;v) \in \eta_x$, then taking $k = |v|$ and $l =  \max\{|f|, |v| :  f\in F\}$ gives us that $\klclass{k}{[x]}{l}\subset C(F;v)$, and so $E(F;v)\in \theta(\tilde x)$. Conversely, if $E(F;v)\in \theta(\tilde x)$, then $\klclass{k}{[x]}{l}\subset C(F;v)$ for some $(k,l)\in \I$. Hence $x\in C(F;v)$,  $E(F;v)\in \eta_x$, and so $\theta(\tilde x) = \eta_x$. 

So $\theta: \tX\to \Ef(\s_\X)$ is continuous, injective, and maps a dense subspace of $\tX$ bijectively onto a dense subspace of $\Et(\s_\X)$. Both $\tX$ and $\Ef(\s_\X)$ are second countable, and so $\theta(\tX) \subset \Et(\s_\X)$. Since $\tX$ is compact, we must have that $\theta(\tX)$ is a closed set in $\Ef(\s_\X)$ which contains $\Eu(\s_\X)$, and so it contains its closure $\Et(\s_\X)$. Therefore $\theta(\tX) = \Et(\s_\X)$, and since $\tX$ is compact and $\Et(\s_\X)$ is Hausdorff, $\theta:\tX\to \Et(\s_\X)$ is a homeomorphism.
\end{proof}
Now that we have the above homeomorphism, we can establish the conditions we need to use the universal property of $\OX$.
\begin{prop}\label{DXEtiso}
There exists a $*$-isomorphism $\Psi:\DX\to C(\Et(\s_\X))$ such that $\Psi(1_{C(w,v)}) = 1_{D^t_{E(w,v)}}$ for all $w, v\in\A^*$. Furthermore, if $E(F;v)\in E_\X$, then $\Psi(1_{C(F;v)}) = 1_{D^t_{E(F;v)}}$.
\end{prop}
\begin{proof}
Take $w,v\in\A^*$, and let $k = |v|$, $l = \max\{|w|,|v|\}$. There are only finitely many $l$-past equivalence classes, so pick a representative from each one, say $x^l_1, x^l_2, \dots x^l_{m(l)}$. By Lemma \ref{eqclasscontain}, $C(w,v)$ is a finite disjoint union of $\kl{k}{l}$ equivalence classes, that is there exists $F\subset \{1, \dots, m(l)\}$ such that 
\begin{eqnarray}
C(w,v) &=& \bigcup_{f\in F}\ \klclass{k}{[vx^l_f]}{l}\nonumber\\
       &=& \bigcup_{f\in F} C(v)\cap \sigma^{-k}([x_f^l]_l).\label{Cwvdisjoint}
\end{eqnarray}
From the proof of Proposition \ref{xttight}, if $\theta$ is as in \eqref{thetadef}, we must have that 
\[
\theta^{-1}(D^t_{E(w,v)}) = \bigcup_{f\in F} U(vx^l_f, k, l)
\]
where again this is a disjoint union. Thus, if $\Theta$ is the $*$-isomorphism from $C(\Et(\s_\X))$ to $C(\tX)$ induced by $\theta$, we have
\[
\Theta(1_{D_{E(w,v)}}) = \sum_{f\in F}1_{ U(vx^l_f, k, l)}.
\]
By \cite[Proposition 3 in Chapter 2]{CarlsenThesis}, there exists a $*$-isomorphism $\psi:\DX\to C(\tX)$ such that $\psi(1_{C(v)\cap \sigma^{-|v|}([x^l_f]_l)}) = 1_{U(vx_f^l, |v|, l)}$. By \eqref{Cwvdisjoint} we have
\begin{eqnarray*}
\Theta^{-1}\circ \psi\left( 1_{C(w,v)} \right) &=& \Theta^{-1}\circ \psi\left(1_{\bigcup_{f\in F} C(v)\cap \sigma^{-k}([x_f^l]_l)}\right)\\
 &=& \Theta^{-1}\left(\sum_{f\in F}\psi\left(1_{C(v)\cap \sigma^{-k}([x_f^l]_l)}\right)\right)\\
 &=& \Theta^{-1}\left(\sum_{f\in F}1_{ U(vx^l_f, k, l)}\right)\\
 &=& \Theta^{-1}(\Theta(1_{D^t_{E(w,v)}}))\\
 &=& 1_{D^t_{E(w,v)}}.
\end{eqnarray*}
Hence taking $\Psi = \Theta^{-1}\circ \psi$ verifies the first statement. The second statement follows from the fact that, for all $F\fs \A^*$ and $v\in \A^*$, we have
\[
1_{C(F;v)} = 1_{\cap_{f\in F}C(f, v)} = \prod_{f\in F}1_{C(f,v)}, 
\]
\[
D^t_{E(F;v)} = D^t_{\prod_{f\in F}E(f,v)} = \bigcap_{f\in F}D^t_{E(f,v)}. 
\]
\end{proof}
We now establish what we need to use the universal property of $\Ct(\s_\X)$
\begin{prop}\label{SXOXtight}
Let $\X$ be a one-sided subshift over $\A$. Then the map $\pi: \s_\X \to \OX$ defined by
\[
\pi(s_\alpha E(F;v)s_\beta^*) = S_\alpha S_v\left(\prod_{f\in F}S_f^*S_f\right)S_v^*S_\beta^*,\hspace{1cm} F\fs \A^*; \alpha, \beta, v\in \A^*
\]
\[
\pi(0) = 0
\]
is a tight representation of $\s_\X$.
\end{prop}
\begin{proof}

Because Definition \ref{OXdef}.1 and the relations \eqref{comm1}, \eqref{comm2}, \eqref{ortho} hold in $\OX$, and each $S_\mu$ is a partial isometry, the same computations from Lemma \ref{idempotentlemma} hold in $\OX$. Hence, the products computed in Lemma \ref{productcomputation} hold in $\OX$, and so $\pi$ is a representation of $\s_\X$.

Now suppose we have $X, Y, Z\fs E_\X$ such that $Z$ is a cover of $E_\X^{X,Y}$. Then we know that, for the universal tight representation $\pi_t$, we have
\[
\bigvee_{z\in Z}\pi_t(z) = \prod_{x\in X}\pi_t(x)\prod_{y\in Y}(1-\pi_t(y)).
\]
By Proposition \ref{DXEtiso}, $\pi_t(e) = \Psi\circ \pi(e)$ for all $e\in E_\X$. Thus we have
\begin{eqnarray*}
\bigvee_{z\in Z}\pi_t(z) &=& \prod_{x\in X}\pi_t(x)\prod_{y\in Y}(1-\pi_t(y))\\
\bigvee_{z\in Z}\Psi\circ \pi(z) &=& \prod_{x\in X}\Psi\circ \pi(x)\prod_{y\in Y}(\Psi\circ \pi(1)-\Psi\circ \pi(y))\\
\Psi\left(\bigvee_{z\in Z}\pi(z)\right) &=& \Psi\left(\prod_{x\in X} \pi(x)\prod_{y\in Y}(\pi(1)- \pi(y))\right)\\
\bigvee_{z\in Z}\pi(z) &=& \prod_{x\in X} \pi(x)\prod_{y\in Y}(I_{\OX}- \pi(y))
\end{eqnarray*}
and so, $\pi$ is a tight representation.
\end{proof}

\begin{theo}\label{mainresult}
Let $\X$ be a one-sided subshift over $\A$, let $\s_\X$ be as in \eqref{SXdef}, and let $\OX$ be as in Definition \ref{OXdef}. Then $\Ct(\s_\X)$ and $\OX$ are $*$-isomorphic.
\end{theo}

\begin{proof}
By Proposition \ref{DXEtiso} and the universal property of $\OX$, there exists a $*$-homomorphism $\kappa: \OX\to \Ct(\s_\X)$ such that $\kappa(S_\mu) = \pi_t(s_\mu)$ for all $\mu\in \A^*$. By Proposition \ref{SXOXtight} and the fact that $\Ct(\s_\X)$ is universal for tight representations for $\s_\X$, there exists a $*$-homomorphism $\tau:\Ct(\s_\X)\to \OX$ such that $\tau(\pi_t(s_\mu)) = S_\mu$ for all $\mu\in \A^*$. We therefore must have that $\kappa$ and $\tau$ are inverses of each other, and so $\Ct(\s_\X)$ and $\OX$ are $*$-isomorphic.
\end{proof}

\subsection{$\OX$ as a partial crossed product}

We close with a nice consequence of Theorem \ref{mainresult}. Recall from Lemma \ref{stronglyEunitarylemma} that $\s_\X$ is strongly $E^*$-unitary. Any strongly $E^*$-unitary inverse semigroup $S$ admits a {\em universal group} $\U(S)$, that is there exists an idempotent-pure partial homomorphism $\iota:S\setminus \{0\}\to \U(S)$ such that if every other idempotent-pure partial homomorphism from $S$ factors through $\iota$. We have the following result about strongly $E^*$-unitary inverse semigroups from \cite{MS14}. 

\begin{theo}\label{MStheo}(See \cite[Theorem 5.3]{MS14}) Let $S$ be a countable strongly $E^*$-unitary inverse semigroup. Then there is a natural partial action of $\U(S)$ on $\Et(S)$ such that the partial crossed product $C(\Et(S))\rtimes \U(S)$ is isomorphic to $\Ct(S)$.
\end{theo}

We do not define partial actions or partial crossed products here -- the interested reader is directed to the excellent reference \cite{ExBook}. 

In a preprint version of this work, we concluded the paper by using the above to deduce that $\OX$ could be written as a partial crossed product by the universal group of $\s_\X$. We are grateful to the referee for pointing out that our results allow us to easily see what the universal group is and to say even more about this partial crossed product. In what remains of this paper, we implement the referee's suggestions.

The following Lemma is a consequence of our proof of Lemma \ref{stronglyEunitarylemma}. 

\begin{lem}\label{universalfreegroup}
Let $\X$ be a one-sided subshift over $\A$. Then $\U(\s_\X)$ is isomorphic to $\mathbb{F}_\A$. 
\end{lem}
\begin{proof}
Let $\phi: \s_\X:\to \mathbb{F}_\A$ be the partial homomorphism from the proof of Lemma \ref{stronglyEunitarylemma}. The group $\U(\s_\X)$ is generated by $\iota(s_a)$ for $a\in \A$, so there is a group homomorphism from $\mathbb{F}_\A$ to $\U(\s_\X)$ which sends $a\in\A$ to $\iota(s_a)$. From the definition of $\U(\s_\X)$, there exists a group homomorphism from $\U(\s_\X)$ to $\mathbb{F}_\A$ such that $\iota(s_a) = a$. Therefore, $\U(\s_\X)$ is isomorphic to $\mathbb{F}_\A$. 
\end{proof}

We now have the following from Theorem \ref{MStheo}
\begin{cor}\label{FreePartialCrossedProduct}
Let $\X$ be a one-sided subshift over $\A$, and let $\OX$ be as in Definition \ref{OXdef}. Then there is partial action of $\mathbb{F}_\A$ on $\tilde\X$ such that $\OX \cong C(\tilde\X)\rtimes \mathbb{F}_\A$.  
\end{cor}

At this point we must direct the reader to the recent preprint \cite{ED15} which constructs by hand the partial action from Corollary \ref{FreePartialCrossedProduct}, studies it in detail, and uses it to give necessary and sufficient conditions on $\X$ to guarantee that $\OX$ is simple. The article \cite{ED15} appeared after the first preprint version of this work but before the final version was accepted. Therefore, the result in Corollary \ref{FreePartialCrossedProduct} is original to \cite{ED15}.

As the referee points out, one can say a little more about this partial crossed product. Given a partial action $\theta$ of a group $\Gamma$ on a space $X$, one can always construct a space $\tilde X \supset X$ and a global action $\tilde \theta$ of $\Gamma$ on $\tilde X$ such that the restriction of $\theta$ to $X$ is the original partial action -- this is called the {\em enveloping action} for $\theta$, see \cite{Ab03}. Unfortunately, even if $X$ is Hausdorff,  $\tilde X$ may not be. When $X$ and $\tilde X$ are both locally compact and Hausdorff, then the partial crossed product $C_0(X)\rtimes_\theta \Gamma$ is strongly Morita equivalent to the crossed product $C_0(\tilde X)\rtimes_{\tilde\theta} \Gamma$, see \cite{Ab03} for the details.

In our situation, \cite[Corollary 6.17]{MS14} says that because $\s_\X$ is $F^*$-inverse, the space for the enveloping action for the partial action in \cite{ED15} and Corollary \ref{FreePartialCrossedProduct} is Hausdorff. Therefore we have the following.

\begin{cor}\label{hausdorffglobalization}
Let $\X$ be a one-sided subshift over $\A$, and let $\OX$ be as in Definition \ref{OXdef}. Then there exists a locally compact Hausdorff space $\Omega$ and an action of $\mathbb{F}_\A$ on $\Omega$ such that $\OX$ is strongly Morita equivalent to $C_0(\Omega)\rtimes \mathbb{F}_\A$. 
\end{cor} 

{\bf Acknowledgment:} I am grateful to the referee for an extremely careful reading, and for pointing out that the results of this paper could be strengthened to include Lemma \ref{universalfreegroup}, Corollary \ref{FreePartialCrossedProduct}, and Corollary \ref{hausdorffglobalization}.

\newpage
\begin{center}\Large{Correction to ``Inverse semigroups associated to subshifts''}
\end{center}

	{\bf Abstract:} The results of this paper are not true in the generality stated. In this note, we correct the paper by showing the results hold under an additional assumption, namely Matsumoto's condition \eqref{eq:I}.

\section{Correction to the above paper}
As mentioned above, most of the results in this paper are not correct. The problem stems from Remark~\ref{rmk:incorrectrmk} which states that if $s_\alpha E(F;v)s_\beta^*$ is in standard form, then $\alpha$ and $\beta$ are unique.
This is not true in general, and unfortunately many of the results that follow rely on this statement. For instance, if $\A = \{a\}$ is a singleton then $\X$ is also a singleton and so all the nonempty functions in $\s_\X$ coincide. However, we have $s_as_\epsilon^* = s_{aa}s_\epsilon^*=  s_\epsilon s_a^*$ with all elements in standard form.  

We will show that uniqueness will follow if we assume $\X$ satisfies condition \eqref{eq:I} of Matsumoto (which we define for convenience below \eqref{eq:I}). Then the statements in the above article are true if one assumes that $\X$ satisfies condition \eqref{eq:I} in the following results: Lemma~3.10, Lemma~3.11, Lemma~3.12, Proposition~3.13, Proposition~4.7, Theorem~4.8, Lemma~4.10, Corollary~4.11, and Corollary~4.12.  Note that condition \eqref{eq:I} is {\bf not} required for Lemma~4.2, Lemma~4.3, Lemma~4.4, Proposition~4.5, and Proposition~4.6.

To define \eqref{eq:I}, for $k\geq 0$ and $x\in \X$ we let $\p_k(x)= \{\mu\in \A^*\mid \mu x\in \X, |\mu| = k\}$. For $l\in \NN$, we say that $x, y\in \X$ are {\em $l$-past equivalent} and write $x\sim_l y$ if $\p_k(x) = \p_k(y)$ for all $k\leq l$. Then we say that $\X$ satisfies condition (I) (see \cite[Section~5]{Mats99}) if
\begin{equation}\tag{I}\label{eq:I}
\text{For all }l\in\NN \text{ and for all }x\in\X, \text{ there exists }y\neq x\text{ such that }x\sim_l y.
\end{equation}
An immediate important consequence of \eqref{eq:I} is that if a given $C(F;v)$ is nonempty, then it cannot be a singleton. This consequence features prominently in the proof of Lemma~\ref{lem:lengthprehom}. 
\begin{lem}\label{lem:lengthprehom}
	Let $\X$ be a one-sided subshift over $\A$ which satisfies condition \eqref{eq:I}. Then the map $L: \s_\X^\times\to  \ZZ$ given by $L(s_\alpha E(F; v)s^*_\beta)  = |\beta|-|\alpha| $ is a prehomomorphism.
\end{lem}
\begin{proof}
	We prove here that $L$ is well-defined; once that is established we will have that $L$ is a prehomomorphism from Lemma~\ref{productcomputation} 
	
	Suppose that $\varphi = s_\alpha E(F;v)s_\beta^* = s_\gamma E(G;w)s_\delta^*$ and that $|\delta|-|\gamma| \neq |\beta|-|\gamma|$. Given $x$ in the domain of $\varphi$, we have $x = \beta y = \delta z$ and $\varphi(x) = \alpha y = \gamma z$ for some $y, z\in \X$. Without loss of generality, we can assume $|\beta|\geq|\delta|$.
	
	Since $\beta y = \delta z$ this means that $\beta = \delta\eta$ for some $\eta\in\A^*$, which implies $\eta y = z$. Hence $\alpha y = \gamma\eta y$. Since $|\eta| = |\beta|-|\delta|$ this implies 
	\[
	y = \sigma^{|\gamma|+|\beta| -|\delta|}(\alpha y) = \sigma^{|\alpha|}(\alpha y).
	\]
	By assumption, the two powers on the shift map are unequal, so $\alpha y$ is eventually periodic. There are now two cases based on which power is larger.
	
	\underline{Case 1} $|\alpha| >|\gamma| + |\beta| - |\delta|$: Let $k = |\alpha| -(|\gamma| + |\beta| - |\delta|) \geq 0$. Then $\sigma^k(y) = \sigma^{|\alpha|}(\alpha y) = y$, so that $y$ is periodic with period $k$. But since $\alpha y = \gamma\eta y$ and $k = |\alpha|-|\gamma\eta|$, we must have that $y = \mu^\infty$, where $\mu$ is the suffix of $\alpha$ of length $k$. Thus $x = \beta \mu^\infty$, and the domain of $\varphi$ is the single point $x$. But \eqref{eq:I} implies the domain of $\varphi$ cannot be a singleton, a contradiction.
	
	\underline{Case 2} $|\alpha| <|\gamma| + |\beta| - |\delta|$: in this case taking $ k = |\gamma| + |\beta| - |\delta|- |\alpha|$ again gives that $y$ is periodic with period $k$, and the same reasoning implies $x = \beta\mu^\infty$ where $\mu$ is the suffix of $\beta$ of length $k$. Again this contradicts \eqref{eq:I}. Since both cases end in contradiction, we conclude that $|\beta|-|\alpha| = |\delta|-|\gamma|$ so that $L$ is well-defined.
\end{proof}
\begin{lem}\label{lem:std_form_uniqueness}
	Let $\X$ be a one-sided subshift over $\A$ which satisfies condition \eqref{eq:I}. If $s_\alpha E(F;v)s_\beta^* = s_\gamma E(G;w)s_\delta^*$ with both in standard form, then $\alpha = \gamma$ and $\beta = \delta$. Furthermore, for $s_\alpha E(F;v)s_\beta^*$ and $s_\alpha E(G;w)s_\beta^*$ in standard form, we have $s_\alpha E(F;v)s_\beta^* = s_\alpha E(G;w)s_\beta^*$ if and only if $C(F\cup\{\alpha v, \beta v\};v) = C(G\cup\{\alpha w, \beta w\}; w)$. 
\end{lem}
\begin{proof}
	Suppose that $\varphi = s_\alpha E(F;v)s_\beta^* = s_\gamma E(G;w)s_\delta^*$. Given $x$ in the domain of $\varphi$, we have $x = \beta y = \delta z$ and $\varphi(x) = \alpha y = \gamma z$ for some $y, z\in \X$. Without loss of generality, we can assume $|\beta|<|\delta|$. This implies $\delta = \beta\eta$ and $\gamma = \alpha\mu$ for some $\eta,\mu\in\A^*$ with nonzero length, and Lemma~\ref{lem:lengthprehom} implies $|\mu| = |\eta|$. This implies $\beta\eta z = \beta y$ and $\alpha\mu z = \alpha y$, implying that $\mu z = y = \eta z$. Since $\mu$ and $\eta$ have the same length and are supposed to have different terminal letters by assumption, this is a contradiction. 
	
	For the second part, first assume that $s_\alpha E(F;v)s_\beta^* = s_\alpha E(G;w)s_\beta^*$. Then multiplying by $s_\alpha^*$ on the left and $s_\beta$ on the right gives $E(\{\alpha\}, \epsilon)E(F;v)E(\{\alpha\}, \epsilon) = E(\{\beta\},\epsilon) E(G;w)E(\{\beta\}, \epsilon)$. Applying the calculations in the proof of Lemma~\ref{idempotentlemma} then gives $E(F\cup\{\alpha v, \beta v\};v) = E(G\cup\{\alpha w, \beta w\}; w)$, which is equivalent to $C(F\cup\{\alpha v, \beta v\};v) = C(G\cup\{\alpha w, \beta w\}; w)$. Conversely, $C(F\cup\{\alpha v, \beta v\};v) = C(G\cup\{\alpha w, \beta w\}; w)$ similarly implies $s_\alpha^*s_\alpha E(F;v)s_\beta^*s_\beta = s_\alpha^*s_\alpha E(G;w)s_\beta^*s_\beta$, and multiplying both sides on the left by $s_\alpha$ and the right by $s_\beta^*$ gives $s_\alpha E(F;v)s_\beta^* = s_\alpha E(G;w)s_\beta^*$.
\end{proof}

With Lemma~\ref{lem:std_form_uniqueness} in hand, the proofs for all but one of the results listed above go through as written. The exception is Proposition~\ref{SXOXtight}, which requires a bit more justification to prove that $\pi$ is well-defined. 

\begin{proof}[Proof that $\pi$ is well-defined in Proposition~4.7] If $s_\alpha E(F;v)s_\beta^* = s_\alpha E(G;w)s_\beta^*$ with both in standard form, then Lemma~\ref{lem:std_form_uniqueness} implies that $C(F\cup\{\alpha v, \beta v\};v) = C(G\cup\{\alpha w, \beta w\}; w)$. The universal property of $\OX$ then implies that 
	\[
	S_\alpha^*S_\alpha S_v\left(\prod_{f\in F}S_f^*S_f\right)S_v^*S_\beta^*S_\beta = S_\alpha^*S_\alpha S_w\left(\prod_{g\in G}S_f^*S_f\right)S_w^*S_\beta^*S_\beta,
	\]
	from which we have 
	\begin{align*}
	\pi(s_\alpha E(F;v)s_\beta^*)& = S_\alpha S_v\left(\prod_{f\in F}S_f^*S_f\right)S_v^*S_\beta^*\\
	& = S_\alpha S_\alpha^*S_\alpha S_v\left(\prod_{f\in F}S_f^*S_f\right)S_v^*S_\beta^*S_\beta S_\beta^*\\
	& = S_\alpha S_\alpha^*S_\alpha S_w\left(\prod_{g\in G}S_g^*S_g\right)S_w^*S_\beta^*S_\beta S_\beta^*\\
	& = S_\alpha S_w\left(\prod_{g\in G}S_g^*S_g\right)S_w^*S_\beta^*\\
	& = \pi(s_\alpha E(G;w)s_\beta^*).
	\end{align*}
	Thus $\pi$ is well-defined.
\end{proof}

{\bf Acknowledgment:} I am grateful to Toke Meier Carlsen for pointing out that the results of this paper were not true in the generality stated, and suggesting they would be true assuming \eqref{eq:I}.

\bibliographystyle{alpha}
\bibliography{E:/Dropbox/Research/bibtex}{}

{\small 
\textsc{University of Ottawa, Department of Mathematics and Statistics. 585 King Edward, Ottawa, ON, Canada, K1N 6N5} \texttt{cstar050@uottawa.ca} 
}
\end{document}